\documentclass[]{amsproc}

\theoremstyle{definition}

\theoremstyle{remark}

\numberwithin{equation}{section}
\usepackage{graphicx}
\usepackage{latexsym}
\usepackage{amsmath}
\usepackage{amssymb}
\usepackage{amsfonts}
\usepackage{verbatim}
\usepackage{mathrsfs}
\usepackage[modulo]{lineno} 
\usepackage[latin1]{inputenc}
\usepackage{color}
\usepackage[colorlinks,citecolor=blue,urlcolor=blue]{hyperref}

\newtheorem{tm}{Theorem}[section]
\newtheorem{rk}{Remark}[section]
\newtheorem{ap}{Assumption}[section]

\newtheorem{prop}{Proposition}[section]
\newtheorem{lm}{Lemma}[section]
\newtheorem{cor}{Corollary}[section]

\newcommand{\ee}{\mathbb E}
\newcommand{\pp}{\mathbb P}
\newcommand{\nn}{\mathbb N}
\newcommand{\rr}{\mathbb R}
\newcommand{\hh}{\mathbb H}

\newcommand{\CC}{\mathcal C}

\newcommand{\OO}{\mathcal O}
\newcommand{\PP}{\mathcal P}

\newcommand{\OOO}{\mathscr O}
\newcommand{\FFF}{\mathscr F}

\newcommand{\<}{\langle}
\renewcommand{\>}{\rangle}
\allowdisplaybreaks \allowdisplaybreaks[4]

\begin{document}

\linenumbers

\title[Wong--Zakai Approximations of SACE]
{Wong--Zakai Approximations of Stochastic Allen--Cahn Equation}

\author{Zhihui Liu}
\address{Department of Applied Mathematics, 
The Hong Kong Polytechnic University, Hung Hom, Kowloon,
Hong Kong}
\curraddr{}
\email{liuzhihui@lsec.cc.ac.cn}
\thanks{}
\author{Zhonghua Qiao}
\address{Department of Applied Mathematics, 
The Hong Kong Polytechnic University, Hung Hom, Kowloon,
Hong Kong}
\curraddr{}
\email{zhonghua.qiao@polyu.edu.hk}
\thanks{The second author is partially supported by Hong Kong RGC General Research Fund, No. 15325816, and the Hong Kong Polytechnic University Research Fund, No. G-YBKP}



\date{\today}

\dedicatory{}

\begin{abstract}
We establish a unconditional and optimal strong convergence rate of Wong--Zakai type approximations in Banach space norm for a parabolic stochastic partial differential equation with monotone drift, including the stochastic Allen--Cahn equation, driven by an additive Brownian sheet.
The key ingredient in the analysis is the fully use of additive nature of the noise and monotonicity of the drift to derive a priori estimation for the solution of this equation, in combination with the factorization method and stochastic calculus in martingale type 2 Banach spaces applied to deduce sharp error estimation between the exact and approximate Ornstein--Uhlenbeck processes, in Banach space norm.
\end{abstract}

\maketitle


\section{Introduction}
\label{sec1}

Consider the following parabolic stochastic partial differential equation (SPDE) driven by an additive Brown sheet $W$: 
\begin{align}\label{ac} 
\begin{split}
&\frac{\partial u(t,x)}{\partial t} 
=\frac{\partial^2 u(t,x)}{\partial x^2} +f(u(t,x))
+\frac{\partial^2 W(t,x)}{\partial t \partial x},
\quad (t,x)\in (0,T]\times (0,1),
\end{split}
\end{align}
with initial value and Dirichlet boundary conditions:
\begin{align}\label{dbc}
u(t,0)=u(t,1)=0, \
u(0,x)=u_0(x), 
\quad (t,x)\in [0,T]\times (0,1).
\end{align}
Here $f$ satisfies certain monotone and polynomial growth conditions (see Assumption \ref{ap-f}).
We remark that if $f(x)=x-x^3$, then Eq. \eqref{ac} is called the stochastic Allen--Cahn equation.
This type of stochastic equation, arisen from phase transition in materials science by stochastic perturbation such as impurities of the materials, has been extensively studied in the literatures; see, e.g., \cite{Cer03(PTRF)}, \cite{KKL07(IFB)} and \cite{Fun95(PTRF)} for one-dimensional white noises and \cite{KLL15(JAP)}, \cite{FLP14(SPDE)} and \cite{FLZ17(SINUM)} for possibly high-dimensional colored noises.

The main concern in this paper is to derive a unconditional and optimal strong convergence rate of Wong--Zakai--Galerkin approximations to simulate the Brownian sheet in Eq. \eqref{ac}.
Specifically, we simulate the space-time white noise by temporal piecewise constant approximation and then make spectral projection to this temporal approximation (see Eq. \eqref{wz}).
This type of approximations and its versions, such as spatio-temporal piecewise constant approximation, have been investigated by many researchers in mathematical and numerical settings.
See, e.g., \cite{BMS95(AOP)} and \cite{FIZ07(AOP)} for mathematical applications to support theorem in H\"older norm for parabolic SPDEs and the existence of stochastic flow for a  stochastic differential equation without Lipschitz conditions;
see, e.g., \cite{DZ02(SINUM)}, \cite{CHL17(SINUM)} and \cite{WLW18(JDE)} for numerical applications to construct Galerkin approximations for SPDEs with Lipschitz coefficients and the convergence of the Wong--Zakai approximate attractors to the original attractor of stochastic reaction-diffusion equations.

We note that the same simulation method had been used in \cite{GKN02(CMP)} for the stochastic Burgers equation, where the authors derived the strong convergence of the proposed simulation method without any algebraic rate. 
On the other hand, the authors in \cite{KKL07(IFB)} regularized the white noise by a spatio-temporal Wong--Zakai approximations and apply to a practical Monte--Carlo method combined with a Euler--Galerkin scheme for the stochastic Allen--Cahn equation.
They \cite[Theorem 4.4]{KKL07(IFB)} used a probabilistic maximum principle which leads to the assumption that $u_0 \in L^\infty(0,1)$ to prove the conditional convergence rate
\begin{align*}
\Big(\ee\Big[\chi_{\Omega_{\tau,h}} \|u-\widehat u\|^2_{L^2((0,T)\times (0,1))} \Big] \Big)^\frac12
=\OO\Big(\tau^\frac14+h/{\tau^\frac14} \Big),
\end{align*}
in a large subset $\Omega_{\tau,h}\subset \Omega$ such that 
$\pp(\Omega_{\tau,h})\rightarrow 1$ as the temporal and spatial step sizes $\tau,h$ tend to 0, where $u$ and $\widehat u$ denote the exact and Wong--Zakai approximate solution of the stochastic Allen--Cahn equation, respectively.

These problems are main motivations for this study to give a unconditional and optimal strong convergence rate of Wong--Zakai-type approximations of Eq. \eqref{ac} with monotone drift which grows polynomially.
Our approach shows that, to derive a strong convergence rate of the proposed Wong--Zakai--Galerkin approximations under the $L^\infty(0,T; L^2(\Omega; L^2(0,1)))$-norm, it is necessary to bound the exact solution and derive the strong convergence rate of the associated exact and approximate Ornstein--Uhlenbeck processes in the $L^p(0,T; L^p (\Omega; L^p (0,1)))$-norm and the $L^\infty(0,T; L^l (\Omega; L^l (0,1)))$-norm, respectively, for possibly large indices $p,l>2$ (see \eqref{emn}).
This is mainly due to the appearance of the polynomial growth in the nonlinearity and quite different from that of \cite{DZ02(SINUM)} and \cite{CHL17(SINUM)} where these authors only needed to deal with the $L^\infty(0,T; L^2(\Omega; L^2(0,1)))$-norm.

To derive the aforementioned a priori estimation for the solution of Eq. \eqref{ac}, the key ingredient in our analysis is by making fully use of the additive nature of the noise which allows the transformation of Eq. \eqref{ac} to the equivalent random partial differential equation (PDE) \eqref{z} and the monotonicity of $f$ (see Proposition \ref{moment}).
Then we combine the factorization method with stochastic calculus in martingale type 2 Banach spaces to bounded uniformly the exact and approximate Ornstein--Uhlenbeck processes and derive a sharp strong convergence rate for them in Banach setting (see Lemma \ref{ou} and Theorem \ref{ou-err}).

The main result is the following unconditional strong convergence rate of the aforementioned Wong--Zakai--Galerkin approximations applied to Eq \eqref{ac}:
\begin{align} \label{wz-err0} 
\sup_{t\in [0,T]} \Big( \ee\Big[\|u(t)-u^{m,n}(t)\|^p_{L^p(0,1)} \Big] \Big)^\frac1p
=\OO\Big[\Big(\frac1m\Big)^\frac14 \wedge \Big(\frac1n\Big)^\frac12 \Big],
\end{align}
for any $1\le p<\frac{p_*}2+1$ provided that $u_0\in L^{p_*}(\Omega; L_x^{p_*})$ (see Theorem \ref{main} and Remark \ref{rk-ac}).
Here $m$, $n$ are the number of temporal steps and dimension of spectral Galerkin space, and $u$ and $u^{m,n}$ denote the exact solution of Eq. \eqref{ac} and the Wong--Zakai--Galerkin approximate solution of Eq. \eqref{wz}, respectively.
Note that we generalize, in a separate paper \cite{LQ17-2}, the approach of the present paper in combination with new technique to derive a strong convergence rate of a fully discrete approximation for Eq. \eqref{ac} under certain regularity condition on the initial datum.

The rest of this article is organized as follows.
Some frequently used notations and preliminaries of stochastic calculus in martingale type 2 Banach setting are given in the next section, there we derive a priori estimation for the solution of Eq. \eqref{ac}.
Finally, we deduce the optimal strong convergence rate for the Wong--Zakai--Galerkin approximation \eqref{wz} of Eq. \eqref{ac} in the last section.

\section{Preliminaries}
\label{sec2}

In this section, we give some commonly used notations and preliminaries of the stochastic calculus in martingale type 2 Banach setting, as well as a priori estimation for the solution of Eq. \eqref{ac}.

\subsection{Notations and Assumption}

Let $p\ge 1$, $r\in [1,\infty]$, $q\in [2,\infty]$, $\theta\ge 0$ and $\delta,\kappa\in [0,1]$.
Here and after we denote by $L_x^q:=L_x^q(0,1)$ and $\hh:=L_x^2$.
Similarly, $L_\omega^p$ and $L_t^r$ denote the related Lebesgue spaces on $\Omega$ and $[0,T]$, respectively.
For convenience, sometimes we use the temporal, sample path and spatial mixed norm $\|\cdot\|_{L_\omega^p L_t^r L_x^q}$ in different orders, such as
\begin{align*}
\|X\|_{L_\omega^p L_t^r L_x^q}
:=\bigg(\int_\Omega \bigg(\int_0^T \bigg(\int_0^1 |u(t,x,\omega)|^q {\rm d}x\bigg)^\frac rq {\rm d}t\bigg)^\frac pr {\rm d}\pp(\omega)\bigg)^\frac 1p
\end{align*}
for $u\in L_\omega^p L_t^r L_x^q$,
with the usual modification for $r=\infty$ or $q=\infty$.

Denote by $A$ the Dirichlet Laplacian on $L_x^q$ for $q\ge 2$.
Then $A$ is the infinitesimal generator of an analytic $C_0$-semigroup $S(\cdot)$ on $L_x^q$, and thus one can define the fractional powers $(-A)^\theta$ for $\theta\in \rr$ of the self-adjoint and positive definite operator $-A$.
Let $\theta\ge 0$ and $\mathbb W_x^{\theta,q}$ be the domain of $(-A)^{\theta/2}$ equipped with the norm $\|\cdot\|_{\mathbb W_x^{\theta,q}}$ (denote $\dot \hh_x^\theta:=\mathbb W_x^{\theta,2}$ and 
$\|\cdot\|_\theta:=\|\cdot\|_{\mathbb W_x^{\theta,q}}$):
\begin{align*}
\|X\|_{\mathbb W_x^{\theta,q}}
:=\|(-A)^\frac\theta2 X\|_{L_x^q},
\quad X\in \mathbb W_x^{\theta,q}.
\end{align*}

For a Banach space $(B,\|\cdot\|_B)$ and a bounded subset 
$\OOO\subset \rr$, we use $\CC(\OOO; B)$ to denote the Banach space consisting of $B$-valued continuous functions $f$ such that
$\|f\|_{\CC(\OOO; B)}:=\sup_{x\in \OOO} \|f(x)\|_B<\infty$, and 
$\CC^\kappa(\OOO; B)$ with $\kappa\in (0,1]$ to denote the $B$-valued function $f$ such that 
\begin{align*}
\|f\|_{\CC^\kappa(\OOO; B)}
:=\sup_{x \in \OOO} \|f(x)\|_B+\sup_{x,y\in \OOO, x\neq y}\frac{\|f(x)-f(y)\|_B}{|x-y|^\kappa}<\infty.
\end{align*}
In the following we simply denote
$\CC^\kappa([0,1]; \rr)=\CC^\kappa$.
Similarly, we use $L^p(\Omega;\CC([0,T]; B))$ to denote the Banach space consisting of $B$-valued a.s. continuous stochastic processes $u=\{u(t):\ t\in [0,T]\}$ such that 
\begin{align*}
\|X\|_{L^p(\Omega;\CC([0,T]; B))}
:=\bigg(\ee\bigg[\sup_{t\in [0,T]} \|u(t)\|_B^p \bigg]\bigg)^\frac1p
<\infty,
\end{align*}
and 
$L^p(\Omega;\CC^\delta([0,T]; B))$ with $\delta\in (0,1]$ to denote $B$-valued stochastic processes $u=\{u(t):\ t\in [0,T]\}$ such that 
\begin{align*}
\|X\|_{L^p(\Omega;\CC^\delta([0,T]; B))}
: & =\bigg(\ee\bigg[\sup_{t\in [0,T]} \|u(t)\|_B^p\bigg]\bigg)^\frac1p \\
&\quad +\bigg(\ee\bigg[\bigg(\sup_{t,s\in [0,T], t\neq s}\frac{\|u(t)-u(s)\|_B}{|t-s|^\delta}\bigg)^p\bigg]\bigg)^\frac1p<\infty.
\end{align*}

Throughout we assume that the drift coefficient $f$ of Eq. \eqref{ac} satisfies the following condition.

\begin{ap} \label{ap-f}
$f$ is continuously differentiable and there exist constants $b\in \rr$, $L_f,\widetilde{L_f}\in \rr_+$ and $q\ge 2$ such that
\begin{align}
(f(x)-f(y)) (x-y) \le b |x-y|^2-L_f |x-y|^q,&
\quad x, y\in \rr;  \label{con-f} \\
 |f(0)|<\infty,\quad
|f'(x)| \le \widetilde{L_f} (1+|x|^{q-2}),&
\quad x \in \rr.  \label{con-f'}
\end{align}
\end{ap}

It is clear from \eqref{con-f'} that $f$ grows as most polynomially of order $(q-1)$ by mean value theorem:
\begin{align} \label{con-f1}
|f(x)| \le C(1+|x|^{q-1}),
\quad x\in \rr,
\end{align}
where $C=C(|f(0)|,\widetilde{L_f})$ is a positive constant.
Here and what follows we use $C$ to denote a universary constant independent of various discrete parameters which may be different in each appearance.
A motivated example of $f$ such that Assumption \ref{ap-f} holds true is a polynomial of odd order $(q-1)$ with negative leading coefficient perturbed with a Lipschitz continuous function (for the stochastic Allen--Cahn equation, $q=4$); see, e.g., \cite[Exmple 7.8]{DZ14}.

\subsection{Stochastic Calculus}

In order to apply the theory of stochastic analysis in Banach setting, we need to transform the original SPDE \eqref{ac} into an infinite dimensional stochastic evolution equation.
To this purpose, let us define $F: L_x^{q'} \rightarrow L_x^q$ by the Nymiskii operators associated with $f$, respectively:
\begin{align*}
F(u)(x):=f(u(x)), \quad u\in L_x^{q'}, ~ x\in [0,1].
\end{align*}
where $q'$ denote the conjugation of $q$, i.e., $1/q'+1/q=1$.
Then by Assumption \ref{ap-f}, 
\begin{align}  
F~\text{has a continuous extension from}~ 
L^{q'}_x ~ \text{to} ~ L^q_x, &  \quad \text{and} \label{F}  \\
_{L^{q'}_x}\<F(x)-F(y), x-y\>_{L^q_x}
\le b \|x-y\|^2-L_f \|x-y\|^q_{L^q_x},  & 
\quad x,y\in L^q_x,  \label{con-F}
\end{align}
where $_{L^{q'}_x} \<\cdot, \cdot\>_{L^q_x}$ denotes the dual between $L^{q'}_x$ and $L^q_x$.
Denote by $W_{\hh}$ the $\hh$-valued cylindrical Wiener process in a stochastic basis $(\Omega,\FFF,(\FFF_t)_{t\in [0,T]},\pp)$, i.e., there exists an orthonormal basis $\{h_k\}_{k=1}^\infty $ of $\hh$ and a sequence of mutually independent Brownian motions $\{\beta_k\}_{k=1}^\infty $ such that 
\begin{align}\label{wiener}
W_{\hh}(t)=\sum_{k=1}^\infty h_k\beta_k(t),\quad t\in [0,T].
\end{align}
Then Eq. \eqref{ac} with initial-boundary value condition \eqref{dbc} is equivalent to the following stochastic evolution equation:
\begin{align}\label{ac} \tag{SACE}
{\rm d}u(t)=(Au(t)+F(u(t))) {\rm d}t+{\rm d}W_{\hh}(t),
\ t\in (0,T];
\quad u(0)=u_0.
\end{align}

Note that for any $q\ge 2$ and $\theta\ge 0$, the function space $\mathbb W_x^{\theta,q}$ is a martingale type $2$ Banach space.
We need the following Burkholder--Davis--Gundy inequality in martingale type $2$ Banach space (see, e.g., \cite[Theorem 2.4]{Brz97(SSR)} and \cite[Section 2]{HHL17}):
\begin{align}\label{bdg}
\bigg\|\int_{0}^t \Phi(r){\rm d}W_{\hh}(r)\bigg\|_{L_\omega^p L_t^\infty L_x^q} 
\le C \big\|\Phi \big\|_{L^p(\Omega; L^2(0,T; \gamma(\hh,L_x^q)))},
\end{align}
for any $\Phi\in L^p(\Omega; L^2(0,T; \gamma(\hh,L_x^q)))$ with $p,q\ge 2$, where $\gamma(\hh,L_x^q)$ denotes the radonifying operator norm:
\begin{align*}
\|\Phi\|_{\gamma(\hh,L_x^q)}
:=\Big\|\sum_{k=1}^\infty\gamma_k \Phi h_k\Big\|_{L^2(\Omega';L_x^q)}.
\end{align*}
Here $\{h_k\}_{k=1}^\infty$ is any orthonormal basis of $\hh$ and  
$\{\gamma_n\}_{n\geq 1}$ is a sequence of independent 
$\mathcal N(0,1)$-random variables on a probability space $(\Omega',\FFF',\pp')$, provided that the above series converges.
We also note that $L_x^q$ with $q\ge 2$ is a Banach function space with finite cotype, then $\Phi\in \gamma(\hh;L_x^q)$ if and only if $(\sum_{k=1}^\infty (\Phi h_k)^2)^{1/2}$ belongs to $L_x^q$ for any orthonormal basis $\{h_k\}_{k=1}^\infty$ of $\hh$; see \cite[Lemma 2.1]{NVW08(JFA)}.
Moreover, in this situation, 
\begin{align}\label{cotype}
\|\Phi\|^2_{\gamma(\hh;L_x^q)}
&\simeq \bigg\|\sum_{k=1}^\infty (\Phi h_k)^2\bigg\|_{L_x^\frac q2},
\quad \Phi\in \gamma(\hh;L_x^q).
\end{align}

\subsection{Ornstein--Uhlenbeck Process}

Recall that a predictable stochastic process $u:[0,T]\times \Omega\rightarrow \hh$  is called a mild solution of Eq. \eqref{ac} if $u\in L^\infty(0,T; \hh)$ a.s.
and it holds a.s. that 
\begin{align} \label{mild}
u(t)=S(t)u_0+\int_0^t S(t-r) F(u(r)){\rm d}r+W_A(t),
\quad t\in [0,T],
\end{align}
where $\{W_A(t):=\int_0^t S(t-r) {\rm d}W_{\hh}(r):\ t\in [0,T]\}$ is the so-called Ornstein--Uhlenbeck process.
The uniqueness of the mild solution of Eq. \eqref{ac} is understood in the sense of stochastic equivalence.
Set $z(t):=u(t)-W_A(t)$, $t\in [0,T]$.
Due to the additive nature, it is clear that $u$ is the unique mild solution of Eq. \eqref{ac} if and only if $z$ is the unique solution of the following random PDE:
\begin{align}\label{z} 
\dot z(t)=A z(t)+F(z(t)+W_A(t)), \quad t\in [0,T]; \quad
z(0)=u_0.
\end{align}

We begin with the following sharp H\"older regularity, and in particular the 
$L_\omega^p L_t^\infty L_x^\infty$-estimation, of the Ornstein--Uhlenbeck process $W_A$.
Our main tool is the following factorization formula, which is valid by stochastic Fubini theorem:
\begin{align*}
\int_0^t S(t-r) {\rm d}W_{\hh}(r)
=\frac{\sin(\pi \alpha)}{\pi} \int_0^t (t-r)^{\alpha-1} S(t-r) W_\alpha(r) {\rm d}r,
\end{align*}
where $\alpha\in (0,1)$ and
$W_\alpha(t):=\int_0^t (t-r)^{-\alpha} S(t-r) {\rm d}W_{\hh}(r)$,
$t\in [0,T]$.
It is known that, when $p>1$ and $1/p<\alpha<1$, the linear operator $R_\alpha$ defined by 
\begin{align*}
R_\alpha f(t)
:=\int_0^t (t-r)^{\alpha-1} S(t-r) f(r) {\rm d} r,\quad t\in [0,T],
\end{align*}
is bounded from $L^p(0,T; L_x^q)$ to $\CC^{\delta}([0,T]; \mathbb W_x^{\theta,q})$ for any $q\ge 2$ with $\delta<\alpha-1/p$ when $\theta=0$ or $\delta=\alpha-1/p-\theta/2$ when 
$\theta>0$ and $\alpha>\theta/2+1/p$;
see, e.g., \cite[Proposition 5.14]{DZ14} or \cite[Proposition 4.1]{HHL17}.

\begin{lm}\label{ou}
For any $p\ge 1$,
$W_A \in L^p(\Omega; \CC^{\delta}([0,T]; \CC^\kappa))$ for any $\delta,\kappa\ge 0$ with 
$\delta+\kappa/2<1/4$.
In particular, there exists a constant $C=C(p)$ such that 
\begin{align} \label{reg-w}
\ee\Big[\sup_{t\in [0,T]} \| W_A(t) \|^p_{L_x^\infty}\Big]
\le C.
\end{align}
\end{lm}

\begin{proof}
Let $p,q\ge 2$.
Applying Fubini theorem and the Burkholder--Davis--Gundy inequality \eqref{bdg} implies that
\begin{align*}
\big\| W_\alpha \big\|^p_{L_\omega^p L_t^p L_x^q}  
&=\int_0^T \ee\bigg[\bigg\| \int_0^t (t-r)^{-\alpha} S(t-r) {\rm d}W_{\hh}(r)\bigg\|_{L_x^q} ^p \bigg]  {\rm d}t  \\
&\le C \int_0^T \bigg( \int_0^t r^{-2\alpha} 
\|S(r)\|^2_{\gamma(\hh;L_x^q)} {\rm d} r \bigg)^\frac p2 {\rm d}t.
\end{align*}
By \eqref{cotype} and the uniform boundedness of $\{e_k=\sqrt 2\sin(k\pi \cdot)\}_{k\in \nn_+}$ (which vanishes on the boundary $0$ and $1$), we have  
\begin{align*}
\|S(t)\|^2_{\gamma(\hh;L_x^q)}
&\simeq \bigg\|\sum_{k=1}^\infty (S(t) e_k)^2\bigg\|_{L_x^\frac q2}
\le \sum_{k=1}^\infty e^{-2\lambda_j t} \|e_k\|^2_{L_x^q}
\le C t^{-\frac12},  \quad t\in (0,T],
\end{align*}
where we have used the elementary inequality $\sum_{k=1}^\infty e^{-2\lambda_j t} \le C t^{-\frac12}$.
Then 
\begin{align*}
\big\| W_\alpha \big\|_{L_\omega^p L_t^p  L_x^q} 
&\le C \bigg(\int_0^T \bigg( \int_0^t r^{-(2\alpha+\frac12)} {\rm d} r \bigg)^\frac p2 {\rm d}t\bigg)^\frac1p,
\end{align*}
which is finite if and only if $\alpha\in (0,1/4)$.
As a result of the H\"older continuity characterization, $W_A \in L^p(\Omega; \CC^{\delta}([0,T]; \mathbb W_x^{\theta,q}))$ for any $\delta,\theta\ge 0$ with 
$\delta+\theta/2<1/4$.
We conclude by the Sobolev embedding
$\mathbb W_x^{\theta,q} \hookrightarrow \CC^\kappa$ with 
$\kappa\in [0,\theta-1/q)$ and taking $q$ sufficiently large.
\end{proof}

\subsection{A Priori Moments' Estimation}

The existence of a unique mild solution of Eq. \eqref{mild} which belongs to $\CC([0,T]; \hh)\cap L^q((0,T)\times (0,1))$ a.s. under the conditions \eqref{F}-\eqref{con-F}, and thus Eq. \eqref{ac} under Assumption \ref{ap-f}, had been established in \cite[Theorem 7.17]{DZ14}.
In the following we give a priori estimation on the moments of this solution, which plays a key role in our analysis.
A weak moments' estimation had been given in \cite[Theorem
4.8]{Da04} for Eq. \eqref{ac} (with non-random initial datum) where $f$ is a polynomial  whose derivative is nonpositive perturbed by a linear function (see \cite[Hypothesis 4.4]{Da04}), i.e., $f(x)=\lambda x-\sum_{k=1}^K a_{2k+1} x^{2k+1}$, $x\in \rr$, with $\lambda\in \rr$ and $a_{2k+1}>0$, 
$k=1,\cdots,K\in \nn_+$.

\begin{prop}\label{moment} 
Let $p\ge 2$ and Assumption \ref{ap-f} hold.
Assume that $u_0\in L^p(\Omega; L_x^p)$.
Then Eq. \eqref{ac} exists a unique mild solution $u=\{u(t):\ t\in [0,T]\}$ which is in $L^p(\Omega; \CC(0,T; L_x^p))
\cap L^{p+q-2}(\Omega; L^{p+q-2}(0,T; L_x^{p+q-2}))$ such that 
\begin{align}\label{moment0}
\ee\Big[\sup_{t\in [0,T]} \|u(t)\|^p_{L^p_x} \Big]
+\int_0^T \ee\Big[ \|u(t)\|^{p+q-2}_{L^{p+q-2}_x} \Big] {\rm d}t 
\le C \Big(1+\ee\Big[\|u_0\|^p_{L^p_x}\Big] \Big).
\end{align}
\end{prop}

\begin{proof}
Let $t\in [0,T]$.
Testing both sides of Eq. \eqref{z} by $|z|^{p-2} z$ and integrating by parts yield that
\begin{align*}
& \frac1p \|z(t)\|^p_{L^p_x}
+(p-1) \int_0^t \< |z(r)|^{p-2}, |\nabla z(r)|^2 \> {\rm d}r \\
&=\frac1p \|u_0\|^p_{L^p_x}
+\int_0^t \<(F(u(r)), |z(r)|^{p-2} z(r)\> {\rm d}r.
\end{align*}
It follows from the condition \eqref{con-f} and Young inequality that
\begin{align*}
& \int_0^t \<(F(u(r)), |z(r)|^{p-2} z(r) \> {\rm d}r \\
&=\int_0^t \<(F(z(r)+W_A(r))-F(W_A(r)), |z(r)|^{p-2} z(r)\> {\rm d}r \\
&\quad -\int_0^t \< W_A(r), |z(r)|^{p-2} z(r)\> {\rm d}r \\
&\le C \int_0^t \Big(\|z(r)\|^p_{L^p_x}+\|W_A(r)\|^p_{L^p_x} \Big) {\rm d}r
-L_f \int_0^t \|z(r)\|^{p+q-2}_{L^{p+q-2}_x}  {\rm d}r.
\end{align*}
Thus we obtain
\begin{align}  \label{est-y}
& \frac1p \|z(t)\|^p_{L^p_x}
+L_f \int_0^t \|z(r)\|^{p+q-2}_{L^{p+q-2}_x} {\rm d}r  \nonumber \\
&\le \frac1p \|u_0\|^p_{L^p_x}
+C \int_0^t \|z(r)\|^p_{L^p_x} {\rm d}r
+C \int_0^t \|W_A(r)\|^p_{L^p_x} {\rm d}r.
\end{align}
Now taking $L^{p/p}_\omega L^\infty_t$-norm, we conclude from the estimation \eqref{reg-w} and H\"older and Gr\"onwall inequalities that
\begin{align*}
\ee\Big[\sup_{t\in [0,T]} \|z(t)\|^p_{L^p_x} \Big]
+\ee\bigg[\int_0^T \|z(t)\|^{p+q-2}_{L^{p+q-2}_x} {\rm d}t \bigg]
\le C \Big(1+\ee\Big[\|u_0\|^p_{L^p_x}\Big] \Big).
\end{align*}
This estimation, in combination with the fact that $u=z+W_A$ and the estimation \eqref{reg-w}, shows \eqref{moment0}.
\end{proof}

\begin{rk} 
Using the arguments in Proposition \ref{moment}, one can also show the 
$L^p(\Omega; \CC(0,T; L_x^\rho))\cap L^{p(\rho+q-2)/\rho}(\Omega; L^{\rho+q-2}(0,T; L_x^{\rho+q-2}))$-well-posedness of Eq. \eqref{ac} for any $p\ge \rho \ge 2$, provided $u_0\in L^p(\Omega; L_x^\rho)$ and Assumption \ref{ap-f} hold.
Moreover, the following estimation holds true:
\begin{align*}
\ee\Big[\sup_{t\in [0,T]} \|u(t)\|^p_{L^\rho_x} \Big]
+\ee\bigg[\bigg(\int_0^T \|u(t)\|^{\rho+q-2}_{L^{\rho+q-2}_x} {\rm d}t\bigg)^\frac p\rho \bigg]
\le C \Big(1+\ee\Big[\|u_0\|^p_{L^\rho_x}\Big] \Big).
\end{align*}
\end{rk}

\section{Wong--Zakai--Galerkin Approximations}
\label{sec3}

This section is devoted to establish the optimal strong convergence rate for Wong--Zakai type approximations.

Let $m,n\in \nn_+$.
Let $\{I_i:=(t_i,t_{i+1}]:\ i=0,1,\cdots,m-1\}$ be an equal length subdivision of the time interval $(0,T]$, and $\PP_n$ denote the orthogonal projection from $\hh$ to its finite dimensional subspace $V_n:=\text{span}\{e_1,e_2,\cdots,e_n\}$, where $\{e_k(\cdot)=\sqrt 2\sin(k\pi \cdot)\}_{k=1}^n$ are the $n$ eigenvectors corresponding to the first eigenvalues 
$\{\lambda_j=(k\pi)^2\}_{k=1}^n$ of the Dirichlet Laplacian $A$.

Let $\beta_k^m$ be the piecewise linear approximation
\begin{align} \label{wz}
\beta_k^m(t)=\beta_k(t_i)+\frac mT (\beta_k(t_{i+1})-\beta_k(t_i))(t-t_i),
\quad t \in I_i,
\end{align}
with initial datum $\beta_k^m(0)=0$, $i=0,1,\cdots,m-1$.
Since $W_{\hh}$ can be formally represented as \eqref{wiener}, the resulting approximation of $W_{\hh}$ can be formally given by
\begin{align} \label{wz}
W_{\hh}^m(t)=W_{\hh}(t_i)+\frac mT (W_{\hh}(t_{i+1})-W_{\hh}(t_i))(t-t_i),
\quad  t \in I_i.
\end{align}

Denote by $u^{m,n}$ the mild solution of 
\begin{align} \label{wz}
 {\rm d}u^{m,n}(t)=(Au^{m,n}(t)+F(u^{m,n}(t))) {\rm d}t
+\PP_n {\rm d}W_{\hh}^m(t), \ t\in (0,T];   \quad 
u^{m,n}(0)=u_0.
\end{align}
Then the related approximate Ornstein--Uhlenbeck process is 
\begin{align*}
W_A^{m,n}(t):=\int_0^t S(t-r)\PP_n {\rm d} W_{\hh}^m(r),
\quad t\in (0,T].
\end{align*}
Since $W^m$ is piecewise linear and therefore of bounded variation, $W_A^{m,n}$ is indeed a classical Riemann-Stieltjes integral whose sample paths can be simulated:
\begin{align} \label{ou-mn}
W_A^{m,n}(t)
=\frac mT \sum_{i=0}^{m-1} \int_{I_i} \bigg[S(t-r) \PP_n \int_{I_i} 
 {\rm d}W_{\hh}(\tau) \bigg] {\rm d}r,
\quad t \in [0,T]. 
\end{align}
Here and in the rest part of the paper we set $S(t-r)=0$ for any $0\le t<r\le T$ to lighten the notations.

We note that such simulation method had been studied in \cite[Lemma 2.2]{GKN02(CMP)} where the authors derived strong convergence in $L_\omega^p L_{t,x}^\infty$-norm for any $p\ge 1$ but without any algebraic rate: 
\begin{align*}
\lim_{m,n\rightarrow \infty} \|W_A-W_A^{m,n}\|_{L_\omega^p L_{t,x}^\infty}=0.
\end{align*}
The following result shows the strong error estimation, between the exact and approximate Ornstein--Uhlenbeck processes, under a weak $L_t^\infty L_\omega^p L_x^q$-norm for any $p\ge 1$ and $q\ge 1$.

\begin{tm} \label{ou-err}
Let $p\ge 1$ and $\rho \ge 1$. 
There exists a constant $C=C(p,\rho)$ such that 
\begin{align} \label{ou-err0} 
\|W_A-W_A^{m,n}\|_{L_t^\infty L_\omega^p L_x^\rho}
\le C \Big[\Big(\frac1m\Big)^\frac14 \wedge \Big(\frac1n\Big)^\frac12 \Big].
\end{align}
\end{tm}

\begin{proof}
Due to the monotonicity of the $L^p$-space with respect to $p$, to prove \eqref{ou-err0} for any $p\ge 1$ and $\rho \ge 1$ it suffices to show \eqref{ou-err0} for any $p=\rho=2k$ which is a even number.

Fix $t \in [0,T]$.
By stochastic Fubini theorem, the approximate Ornstein--Uhlenbeck process $W_A^{m,n}$ in \eqref{ou-mn} can be rewritten as 
\begin{align*}
W_A^{m,n}(t)
=\frac mT \sum_{i=0}^{m-1} \int_{I_i} \bigg[\int_{I_i} S(t-\tau) \PP_n
{\rm d}\tau \bigg] {\rm d}W_{\hh}(r).
\end{align*}
Then we have
\begin{align*}
& \ee\bigg[ \|W_A(t)-W_A^{m,n}(t)\|^{2k}_{L_x^{2k}}\bigg] \\
&= \ee\bigg[ \bigg\|\frac mT \sum_{i=0}^{m-1} \int_{I_i} \bigg[\int_{I_i} (S(t-r)-S(t-\tau)\PP_n) {\rm d}\tau \bigg] {\rm d}W_{\hh}(r) \bigg\|^{2k}_{L_x^{2k}} \bigg] \\
&=\Big(\frac mT \Big)^{2k} \int_0^1 
\ee\bigg[ \bigg| \sum_{i=0}^{m-1} \int_{I_i} \bigg[\int_{I_i} (S(t-r)-S(t-\tau) \PP_n ) 
{\rm d}\tau \bigg] {\rm d}W_{\hh}(r) \bigg|^{2k} \bigg] {\rm d}x.
\end{align*}

It is not difficult to show that 
\begin{align} \label{claim}
\ee\bigg[ \bigg| \sum_{i=0}^{m-1} a_i \bigg|^{2k}\bigg]
= \ee\bigg[ \bigg(\sum_{i=0}^{m-1} |a_i |^2 \bigg)^k \bigg],
\end{align}
for any independent centered random variable $a_i$, $i=0,1\cdots,m-1$.
Due to the independence of the Wiener integral in disjoint temporal intervals, we can use \eqref{claim} with $a_i=\int_{I_i} \big[\int_{I_i} (S(t-r)-S(t-\tau) \PP_n ) {\rm d}\tau \big] {\rm d}W_{\hh}(r)$, $i=0,1,\cdots,m-1$, and Minkovskii inequality to deduce that 
\begin{align*}
& \ee\bigg[ \|W_A(t)-W_A^{m,n}(t)\|^{2k}_{L_x^{2k}}\bigg] 
=\Big(\frac mT \Big)^{2k} \int_0^1 
\ee\bigg[ \bigg( \sum_{i=0}^{m-1} |a_i |^2 \bigg)^k \bigg] {\rm d}x  \\
&=\Big(\frac mT \Big)^{2k} 
\bigg\| \sum_{i=0}^{m-1} |a_i |^2 \bigg\|^k_{L_{\omega, x}^{k}} 
\le \Big(\frac mT \Big)^{2k} 
\bigg( \sum_{i=0}^{m-1} \|a_i \|^2_{L_{\omega, x}^{2k}} \bigg)^k.
\end{align*}
It follows from the Burkholder--Davis--Gundy inequality \eqref{bdg}, the estimation \eqref{cotype} and Minkovskii inequality that 
\begin{align*}
& \ee\bigg[ \|W_A(t)-W_A^{m,n}(t)\|^{2k}_{L_x^{2k}}\bigg] \\
&\le C m^{2k} 
\bigg(\sum_{i=0}^{m-1} \int_{I_i} \bigg\|\int_{I_i} (S(t-r)-S(t-\tau)\PP_n ) 
{\rm d}\tau \bigg\|^2_{\gamma(\hh,L_x^{2k})} {\rm d}r \bigg)^k \\
&\le C m^{2k} 
\bigg(\sum_{i=0}^{m-1} \int_{I_i} \sum_{j=1}^\infty \bigg\|\int_{I_i} (S(t-r)-S(t-\tau)\PP_n ) e_j {\rm d}\tau \bigg\|^2_{L_x^{2k}} {\rm d}r \bigg)^k.
\end{align*}
As a result of Minkovskii inequality and Fubini Theorem, we get  
\begin{align*}
\ee\bigg[ \|W_A(t)-W_A^{m,n}(t)\|^{2k}_{L_x^{2k}}\bigg] 
\le C m^{2k} 
\bigg(\sum_{j=n+1}^\infty \sum_{i=0}^{m-1} \Psi_k^i(t) \bigg)^k,
\end{align*}
where 
\begin{align*}
\Psi_k^i(t):=\int_{I_i} \left[\int_{I_i} 
\Big(\chi_{r<t} e^{-\lambda_j(t-r)}-\chi_{\tau<t} e^{-\lambda_j(t-\tau)}\Big)
d\tau \right]^2 {\rm d}r,
\end{align*}
for $t\in [0,T)$ and $i=0,1,\cdots,m-1$.
Here $\chi$ denotes the indicative function, i.e., $\chi_{r<t}=1$ when $r<t$ and vanishes otherwise.

If follows from \cite[Lemma 3.1]{CHL17(SINUM)} that 
\begin{align*}
\sum_{i=0}^{m-1} \Psi_k^i(t) 
\le 8 \Big(\frac1m\Big)^2 \frac{1-e^{\lambda_j m^{-1}}}{\lambda_j},
\end{align*}
from which we get 
\begin{align*}
\sum_{j=n+1}^\infty \sum_{i=0}^{m-1} \Psi_k^i(t) 
\le 8 \Big(\frac1m\Big)^2 \bigg(\sum_{j=n+1}^\infty 
\frac{1-e^{-\lambda_j m^{-1}}}{\lambda_j} \bigg) 
\le C \Big(\frac1m\Big)^2 
\Big[\Big(\frac1m\Big)^\frac12 \wedge \Big(\frac1n\Big)\Big].
\end{align*}
Collecting the above estimations, we conclude \eqref{ou-err0} with $p=\rho=2k$ being a even number and complete the proof
\end{proof}

\begin{rk}
The strong error estimation \eqref{ou-err} is optimal.
The temporal strong convergence rate $\OOO(m^{-1/4})$ under the 
$L_t^\infty L_{\omega,x}^2$-norm had been derived in \cite[(31) of Theorem 3.1]{CHL17(SINUM)} for white noise which is a fractional noise with Hurst parameter $H=1/2$.
To illustrate the optimality of the spatial convergence rate $\OOO(n^{-1/2})$, we use the elementary estimation $e^x\ge 1+x$ for any $x\ge 0$ to show the reverse estimation
\begin{align*}
& \ee\bigg[ \bigg\|\int_0^t S(t-r) {\rm d}W_{\hh}(r)-\int_0^t S(t-r) \PP_n {\rm d}W_{\hh}(r) \bigg\|^2 \bigg] \\
&=\sum_{j=n+1}^\infty \frac{1-e^{-2\lambda_j t}}{2\lambda_j}
\ge \frac t{2(1+2\pi^2 t)} \cdot \frac1n,
\quad t>0.
\end{align*}
\end{rk}

Now we can give the optimal strong convergence rate of the Wong--Zakai--Galerkin approximation \eqref{wz} for Eq. \eqref{ac}.

\begin{tm} \label{main}
Let $p_*\ge 2$, $u_0\in L^{p_*}(\Omega; L_x^{p_*})$ and Assumption \ref{ap-f} hold.
Let $u$ and $u^{m,n}$ be the solutions of Eq. \eqref{ac} and \eqref{wz}, respectively.
Then for any $p\in [1,\frac{p_*}{q-2}+1)$, there exists a constant $C=C(T,p,p_*,b,q,L_f, L_f')$ such that 
\begin{align} \label{main0} 
& \sup_{t\in [0,T]} \ee\Big[ \|u(t)-u^{m,n}(t)\|^p_{L_x^p} \Big]
+\int_0^T \ee\Big[ \|u(t)-u^{m,n}(t)\|^{p+q-2}_{L_x^{p+q-2}} \Big] {\rm d}t \nonumber \\
& \le C \Big(1+\|u_0\|^\frac{p_*p(q-2)}{p_*+q-2}_{L^{p_*}_{\omega,x}} \Big)
\Big[\Big(\frac1m\Big)^\frac14 \wedge \Big(\frac1n\Big)^\frac12 \Big]^p.
\end{align}
\end{tm}

\begin{proof}
Define $z^{m,n}:=u^{m,n}-W_A^{m,n}$.
Then $z^{m,n}$ satisfies 
 \begin{align*}
\dot z^{m,n}(t)=A z^{m,n}(t)+F(z^{m,n}(t)+W_A^{m,n}(t)),\ t\in (0,T]; \   
z^{m,n}(0)=u_0. 
\end{align*}
Let $t\in (0,T]$ and denote by $e^{m,n}(t):=z(t)-z^{m,n}(t)$.
Then 
\begin{align} \label{eq-emn}
\begin{split}
& \dot e^{m,n}(t)=A e^{m,n}(t)+F(z(t)+W_A(t))-F(z^{m,n}(t)+W_A^{m,n}(t)),~ t\in (0,T]; \\
& e^{m,n}(0)=0.
\end{split}
\end{align}

Testing both sides of Eq. \eqref{eq-emn} by $|e^{m,n}(t)|^{p-2} e^{m,n}(t)$ and integrating by parts, similarly to the proof of Proposition \ref{moment}, yield that
\begin{align}  \label{eq-emn1} 
& \frac1p \|e^{m,n}(t)\|^p_{L_x^p}
+(p-1) \int_0^t \<|e^{m,n}|^{p-2}, |\nabla e^{m,n}|^2 \> {\rm d}r \nonumber \\
&=\int_0^t \<F(z+W_A)-F(z+W^{m,n}_A), |e^{m,n}|^{p-2} e^{m,n}\> {\rm d}r 
 \nonumber \\
&\quad +\int_0^t \<F(z+W^{m,n}_A)-F(z^{m,n}+W^{m,n}_A), |e^{m,n}|^{p-2} e^{m,n}\> {\rm d}r.
\end{align}
Using mean value theorem and the conditions \eqref{con-f}-\eqref{con-f'}, and applying Young and H\"older inequalities, we can bound the two terms in the righthand side of the above equality by 
\begin{align*}
&\le C \int_0^t \|F(z+W_A)-F(z+W^{m,n}_A)\|_{L_x^p}^p {\rm d}r
+C \int_0^t \|e^{m,n}\|_{L_x^p}^p {\rm d}r 
\nonumber \\
&\quad -L_f \int_0^t \|e^{m,n}\|_{L_x^{p+q-2}}^{p+q-2} {\rm d}r
\nonumber  \\
&\le C \int_0^t \Big\| \Big(1+|z|^{q-2}+|W_A|^{q-2}+|W_A^{m,n}|^{q-2} \Big) 
|W_A-W^{m,n}_A| \Big \|_{L_x^p}^p {\rm d}r \nonumber  \\
&\quad +C \int_0^t \|e^{m,n}\|_{L_x^p}^p {\rm d}r
-\frac{L_f}2 \int_0^t \|e^{m,n}\|_{L_x^{p+q-2}}^{p+q-2}{\rm d}r.
\end{align*}

Now taking $L^1_\omega L^\infty_t$-norm on both sides of \eqref{eq-emn1}, we have
\begin{align*}
& \ee\Big[\sup_{t\in [0,T]} \|e^{m,n}(t)\|^p_{L^p_x} \Big]
+\int_0^T \ee\Big[ \|e^{m,n}(t)\|^{p+q-2}_{L^{p+q-2}_x} \Big] {\rm d}t 
\\
&\le \Big\| \Big(1+|z|^{q-2}+|W_A|^{q-2}+|W_A^{m,n}|^{q-2} \Big) 
|W_A-W^{m,n}_A| \Big \|_{L_{t,\omega,x}^p}^p   \\
&\quad +C \int_0^T \ee\Big[ \|e^{m,n}(t)\|_{L_x^p}^p \Big] {\rm d}t.
\end{align*}

Let $\epsilon>0$ and denote $p_\epsilon:=\frac{p(p+\epsilon)}\epsilon$ such that $\frac1{p+\epsilon}+\frac1{p_\epsilon}=\frac1p$.
By H\"older inequality, we get 
\begin{align*}
& \Big\| \Big(1+|z|^{q-2}+|W_A|^{q-2}+|W^{m,n}|^{q-2} \Big) 
|W_A-W^{m,n}_A| \Big \|_{L_{t,\omega,x}^p}^p \\
&\le \big\|W_A-W^{m,n}_A \big\|^p_{L_{t,\omega,x}^{p_\epsilon}}
\cdot \Big(1+\|z\|^{p(q-2)}_{L_{t,\omega,x}^{(p+\epsilon)(q-2)}}
+\|W_A\|^{p(q-2)}_{L_{t,\omega,x}^{(p+\epsilon)(q-2)}}
+\|W^{m,n}\|^{p(q-2)}_{L_{t,\omega,x}^{(p+\epsilon)(q-2)}} \Big).
\end{align*}
Since $p\in [1,\frac{p_*}{q-2}+1)$, one can choose 
$0<\epsilon<\frac{p_*}{q-2}+1-p$ such that
$(p+\epsilon)(q-2)<p_*+q-2$.
It follows that 
\begin{align} \label{emn}
& \ee\Big[\sup_{t\in [0,T]} \|e^{m,n}(t)\|^p_{L^p_x} \Big]
+\int_0^T \ee\Big[ \|e^{m,n}(t)\|^{p+q-2}_{L^{p+q-2}_x} \Big] {\rm d}t  \nonumber \\
&\le C \int_0^T \ee\Big[ \|e^{m,n}(t)\|_{L_x^p}^p \Big] {\rm d}t 
+C \|W_A-W^{m,n}_A\|^p_{L_{t,\omega,x}^{p_\epsilon}} \nonumber  \\
&\qquad  \times \Big(1+\|z\|^{p(q-2)}_{L_{t,\omega,x}^{p_*+q-2}}
+\|W_A\|^{p(q-2)}_{L_{t,\omega,x}^{p_*+q-2}}
+\|W^{m,n}\|^{p(q-2)}_{L_{t,\omega,x}^{p_*+q-2}} \Big).
\end{align}
The error estimation of $W_A$ and $W_A^{m,n}$ in Theorem \ref{ou-err}, combining with the regularity of $W_A$ in Lemma \ref{ou} and the estimation \eqref{moment0}, ensures that
\begin{align*}
\|z\|^{p(q-2)}_{L_{t,\omega,x}^{p_*+q-2}}
+\|W_A\|^{p(q-2)}_{L_{t,\omega,x}^{p_*+q-2}}
+\|W^{m,n}\|^{p(q-2)}_{L_{t,\omega,x}^{p_*+q-2}}
\le C \Big(1+\|u_0\|^\frac{p_*p(q-2)}{p_*+q-2}_{L^{p_*}_{\omega,x}} \Big).  
\end{align*}
Substituting the above estimation into \eqref{emn}, we obtain
\begin{align*}
& \ee\Big[\sup_{t\in [0,T]} \|e^{m,n}(t)\|^p_{L^p_x} \Big]
+\int_0^T \ee\Big[ \|e^{m,n}(t)\|^{p+q-2}_{L^{p+q-2}_x} \Big] {\rm d}t  \nonumber \\
&\le C \Big(1+\|u_0\|^\frac{p_*p(q-2)}{p_*+q-2}_{L^{p_*}_{\omega,x}}  \Big)
\big\|W_A-W^{m,n}_A \big\|^p_{L_{t,\omega,x}^{p_\epsilon}}.
\end{align*}
It follows from the relations $u=z+W_A$ and $u^{m,n}=z^{m,n}+W_A^{m,n}$ and triangle inequality that 
\begin{align*}  
& \sup_{t\in [0,T]} \ee\Big[ \|u(t)-u^{m,n}(t)\|^p_{L^p_x} \Big]
+\int_0^T \ee\Big[ \|u(t)-u^{m,n}(t)\|^{p+q-2}_{L^{p+q-2}_x} \Big] {\rm d}t   \\
& \le \|e^{m,n}\|^p_{L_t^\infty L_{\omega,x}^p} 
+\|e^{m,n}\|^{p+q-2}_{L_{t,\omega,x}^{p+q-2}}
+\|W_A-W_A^{m,n}\|^p_{L_t^\infty L_{\omega,x}^p }
+\|W_A-W_A^{m,n}\|^{p+q-2}_{L_{t,\omega,x}^{p+q-2}} \\
&\le  C \Big(1+\|u_0\|^\frac{p_*p(q-2)}{p_*+q-2}_{L^{p_*}_{\omega,x}}\Big)
\Big(\Big\|W_A-W_A^{m,n} \Big\|^p_{L_t^\infty L_{\omega,x}^{p_\epsilon}}
+\Big\|W_A-W_A^{m,n} \Big\|^{p+q-2}_{L_t^\infty 
L_{\omega,x}^{p+q-2}} \Big).   
\end{align*}
Applying Theorem \ref{ou-err}, we get \eqref{main0}.
\end{proof}

\begin{rk} \label{rk-ac}
In the case of stochastic Allen--Cahn equation, i.e., Eq. \eqref{ac} with $f(x)=x-x^3$ for $x\in \rr$, then Assumption \ref{ap-f} holds with $q=4$.
Applying the estimation \eqref{main0} of Theorem \ref{main}, the Wong--Zakai--Galerkin approximation \eqref{wz} applied to this equation possesses the strong convergence rate 
\begin{align*} 
& \sup_{t\in [0,T]} \ee\Big[ \|u(t)-u^{m,n}(t)\|^p_{L_x^p} \Big]
+\int_0^T \ee\Big[ \|u(t)-u^{m,n}(t)\|^{p+2}_{L_x^{p+2}} \Big] {\rm d}t \nonumber \\
& \le C \Big(1+\|u_0\|^\frac{2p_*p}{p_*+2}_{L^{p_*}_{\omega,x}} \Big) \Big[\Big(\frac1m\Big)^\frac p4 \wedge \Big(\frac1n\Big)^\frac p2 \Big],
\end{align*}
for any $2\le p<\frac{p_*}2+1$ provided that $u_0\in L^{p_*}(\Omega; L_x^{p_*})$.
\end{rk}

One can also use a modified argument as in Theorem \ref{main} to derive a strong convergence rate which might not optimal when $q>2$ under minimal assumptions on the initial datum.

\begin{cor}  \label{rk-main}
Let $p\ge 2$, $u_0\in L^p(\Omega; L_x^p)$ and Assumption \ref{ap-f} hold.
Let $u$ and $u^{m,n}$ be the solutions of Eq. \eqref{ac} and \eqref{wz}, respectively.
Then there exists a constant $C=C(T,p,b,q,L_f, L_f')$ such that 
\begin{align} \label{main1} 
& \sup_{t\in [0,T]} \ee\Big[ \|u(t)-u^{m,n}(t)\|^p_{L_x^p} \Big]
+\int_0^T \ee\Big[ \|u(t)-u^{m,n}(t)\|^{p+q-2}_{L_x^{p+q-2}} \Big] {\rm d}t \nonumber \\
& \le C \Big(1+\ee\Big[\|u_0\|^p_{L^p_x}\Big] \Big)
\Big[\Big(\frac1m\Big)^\frac14 \wedge \Big(\frac1n\Big)^\frac12 \Big]^{\frac{p+q-2}{q-1}}.
\end{align}
\end{cor}

\begin{proof}
One only need to modify the proof of Theorem \ref{main} by estimating the term
\begin{align*} 
\ee\bigg[\sup_{t\in [0,T]} \bigg|\int_0^t \<F(z+W_A)-F(z+W^{m,n}_A), |e^{m,n}|^{p-2} e^{m,n}\> {\rm d}r \bigg| \bigg].
\end{align*}
Using mean value theorem and the conditions \eqref{con-f}-\eqref{con-f'}, and applying Young inequality, we have 
\begin{align*}
& \ee\bigg[\sup_{t\in [0,T]} \bigg|\int_0^t \<F(z+W_A)-F(z+W^{m,n}_A), |e^{m,n}|^{p-2} e^{m,n}\> {\rm d}r \bigg| \bigg]  \\
&\le C \int_0^T \|F(z+W_A)-F(z+W^{m,n}_A)\|_{L_{\omega,x}^{\frac{p+q-2}{q-1}}}^{\frac{p+q-2}{q-1}} {\rm d}r
+\frac{L_f}2 \|e^{m,n}\|_{L_{t,\omega,x}^{p+q-2}}^{p+q-2}  \\
&\le C \Big\| \Big(1+|z|^{q-2}+|W_A|^{q-2}+|W_A^{m,n}|^{q-2} \Big) 
|W_A-W^{m,n}_A| \Big \|_{L_{t,\omega,x}^{\frac{p+q-2}{q-1}}}^{\frac{p+q-2}{q-1}} \\
&\quad +\frac{L_f}2 \|e^{m,n}\|_{L_{t,\omega,x}^{p+q-2}}^{p+q-2}. 
\end{align*}
By Young and H\"older inequalities and the estimation \eqref{moment0}, we have
\begin{align*}
& \Big\| \Big(1+|z|^{q-2}+|W_A|^{q-2}+|W_A^{m,n}|^{q-2} \Big) 
|W_A-W^{m,n}_A| \Big \|_{L_{t,\omega,x}^{\frac{p+q-2}{q-1}}}^{\frac{p+q-2}{q-1}}  \\
&\le C \big\|W_A-W^{m,n}_A \big\|^{\frac{p+q-2}{q-1}}_{L_{t,\omega,x}^{p+q-2}}
\Big(1+\|z\|^{p+q-2}_{L_{t,\omega,x}^{p+q-2}}
+\|W_A\|^{p+q-2}_{L_{t,\omega,x}^{p+q-2}}
+\|W^{m,n}\|^{p+q-2}_{L_{t,\omega,x}^{p+q-2}} \Big) \\
&\le C \Big(1+\ee\Big[\|u_0\|^p_{L^p_x}\Big] \Big) 
\big\|W_A-W^{m,n}_A \big\|^{\frac{p+q-2}{q-1}}_{L_{t,\omega,x}^{p+q-2}}.
\end{align*}
Consequently,
\begin{align*}
& \ee\bigg[\sup_{t\in [0,T]} \bigg|\int_0^t \<F(z+W_A)-F(z+W^{m,n}_A), |e^{m,n}|^{p-2} e^{m,n}\> {\rm d}r \bigg| \bigg]  \\
&\le C \Big(1+\ee\Big[\|u_0\|^p_{L^p_x}\Big] \Big) 
\big\|W_A-W^{m,n}_A \big\|^{\frac{p+q-2}{q-1}}_{L_{t,\omega,x}^{p+q-2}}
+\frac{L_f}2 \|e^{m,n}\|_{L_{t,\omega,x}^{p+q-2}}^{p+q-2}. 
\end{align*}
Taking $L^1_\omega L^\infty_t$-norm on both sides of \eqref{eq-emn1} and substituting into the above estimation, we have by Gr\"onwall inequality that 
\begin{align*}
& \ee\Big[\sup_{t\in [0,T]} \|e^{m,n}(t)\|^p_{L^p_x} \Big]
+\int_0^T \ee\Big[ \|e^{m,n}(t)\|^{p+q-2}_{L^{p+q-2}_x} \Big] {\rm d}t  \nonumber \\
&\le C \Big(1+\ee\Big[\|u_0\|^p_{L^p_x}\Big] \Big)
\big\|W_A-W^{m,n}_A \big\|^{\frac{p+q-2}{q-1}}_{L_{t,\omega,x}^{p+q-2}}.
\end{align*}
Following the proof of Theorem \ref{main}, we get 
\begin{align*}  
& \ee\Big[\sup_{t\in [0,T]} \|u(t)-u^{m,n}(t)\|^p_{L^p_x} \Big]
+\int_0^T \ee\Big[ \|u(t)-u^{m,n}(t)\|^{p+q-2}_{L^{p+q-2}_x} \Big] {\rm d}t   \\
&\le  C \Big(1+\ee\Big[\|u_0\|^p_{L^p_x}\Big] \Big)
\Big(\|W_A-W^{m,n}_A\|^{\frac{p+q-2}{q-1}}_{L_{t,\omega,x}^{p+q-2}}
+\|W_A-W_A^{m,n}\|^p_{L_t^\infty 
L_{\omega,x}^{p+q-2}} \Big).   
\end{align*}
Applying Theorem \ref{ou-err}, we get \eqref{main1}.
\end{proof}

\begin{rk}
The assumption on the initial datum, to derive a strong convergence rate between $u$ and $u^{m,n}$ under the $L_{\omega,x}^p$-norm, is minimal.
However, the convergence rate in Remark \ref{rk-main} is far from sharp, since $\frac{p+q-2}{q-1}\le p$ and the equality holds if and only if $q=2$ which reduces to the Lipschitz case.
\end{rk}

\bibliographystyle{amsalpha}
\bibliography{bib}

\end{document}